\documentclass[12pt]{article}
\usepackage{amsmath,amsfonts}
\usepackage{verbatim}
\usepackage{latexsym}
\usepackage{graphicx}
\usepackage{color}
\usepackage{epstopdf}

\makeatletter \@addtoreset{equation}{section}

\begin{document}

\newcommand{\E}{\mathbb{E}}
\newcommand{\PP}{\mathbb{P}}
\newcommand{\RR}{\mathbb{R}}

\newcommand{\Dt}{\D t}
\newcommand{\bX}{\bar X}
\newcommand{\bx}{\bar x}
\newcommand{\by}{\bar y}
\newcommand{\bp}{\bar p}
\newcommand{\bq}{\bar q}

\newtheorem{theorem}{Theorem}[section]
\newtheorem{lemma}[theorem]{Lemma}
\newtheorem{coro}[theorem]{Corollary}
\newtheorem{defn}[theorem]{Definition}
\newtheorem{assp}[theorem]{Assumption}
\newtheorem{expl}[theorem]{Example}
\newtheorem{prop}[theorem]{Proposition}
\newtheorem{rmk}[theorem]{Remark}

\newcommand\tq{{\scriptstyle{3\over 4 }\scriptstyle}}
\newcommand\qua{{\scriptstyle{1\over 4 }\scriptstyle}}
\newcommand\hf{{\textstyle{1\over 2 }\displaystyle}}
\newcommand\athird{{\scriptstyle{1\over 3 }\scriptstyle}}
\newcommand\hhf{{\scriptstyle{1\over 2 }\scriptstyle}}

\newcommand{\eproof}{\indent\vrule height6pt width4pt depth1pt\hfil\par\medbreak}

\def\a{\alpha} \def\g{\gamma}
\def\e{\varepsilon} \def\z{\zeta} \def\y{\eta} \def\o{\theta}
\def\vo{\vartheta} \def\k{\kappa} \def\l{\lambda} \def\m{\mu} \def\n{\nu}
\def\x{\xi}  \def\r{\rho} \def\s{\sigma}
\def\p{\phi} \def\f{\varphi}   \def\w{\omega}
\def\q{\surd} \def\i{\bot} \def\h{\forall} \def\j{\emptyset}

\def\be{\beta} \def\de{\delta} \def\up{\upsilon} \def\eq{\equiv}
\def\ve{\vee} \def\we{\wedge}

\def\F{{\cal F}}
\def\T{\tau} \def\G{\Gamma}  \def\D{\Delta} \def\O{\Theta} \def\L{\Lambda}
\def\X{\Xi} \def\S{\Sigma} \def\W{\Omega}
\def\M{\partial} \def\N{\nabla} \def\Ex{\exists} \def\K{\times}
\def\V{\bigvee} \def\U{\bigwedge}

\def\1{\oslash} \def\2{\oplus} \def\3{\otimes} \def\4{\ominus}
\def\5{\circ} \def\6{\odot} \def\7{\backslash} \def\8{\infty}
\def\9{\bigcap} \def\0{\bigcup} \def\+{\pm} \def\-{\mp}
\def\[{\langle} \def\]{\rangle}

\def\proof{\noindent{\it Proof. }}
\def\tl{\tilde}
\def\trace{\hbox{\rm trace}}
\def\diag{\hbox{\rm diag}}
\def\for{\quad\hbox{for }}
\def\refer{\hangindent=0.3in\hangafter=1}

\newcommand\wD{\widehat{\D}}

\thispagestyle{empty}

\title{
\bf A note on convergence and stability of the truncated Milstein method for stochastic differential equations}

\author{Weijun Zhan\textsuperscript{a}\quad
Yanan Jiang\textsuperscript{b}\thanks{Corresponding author, Email: y.n.jiang@qq.com }\quad
Wei Liu\textsuperscript{b}
\\
\textsuperscript{a}Department of Mathematics, \\Shanghai University, Shanghai, China, 200444\\
\textsuperscript{b}Department of Mathematics, \\Shanghai Normal University, Shanghai, China, 200234
}
\date{}
\maketitle
\begin{abstract}
Some new techniques are employed to release significantly the requirements on the step size of the truncated Milstein method, which was originally developed in Guo, Liu, Mao and Yue (2018). The almost sure stability of the method is also investigated. Numerical simulations are presented to demonstrate the theoretical results.


\medskip \noindent
{\small\bf MAS Classification (2000):}  65C20

\end{abstract}
\section{Introduction}
The classical Milstein method was proposed in \cite{Milstein1974} with the merit of the convergence rate of one, when both the drift and diffusion coefficients of stochastic differential equations (SDEs) satisfy the global Lipschitz and linear growth conditions. However, the classical explicit methods, including the Milstein and Euler-Maruyama methods, are of divergence, when such usual conditions are disturbed \cite{HJK2011a}.
\par
To tackle the super-linearities in the coefficients, one approach is to construct implicit methods. We just mention some of the works here \cite{HMS2013,MS2013a,Mil01,WGW2012} and refer the readers to the references therein.
\par
On the other hand, due to the simple structure and the avoidance of solving non-linear equation systems in each iteration \cite{Hig2011a}, explicit methods still play an important role in the numerical approximates to SDEs. The tamed Euler method was proposed and generalised in \cite{Hut02} and \cite{HJ2015a}. The proof for the method was modified in \cite{Sab2013}. The tamed Milstein methods were developed in \cite{WG2013} and \cite{KS2017a} for SDEs driven by Brownian motion and L\'evy noise, respectively.
\par
Another explicit method, called the truncated Euler method, was originally proposed in \cite{Mao2015,Mao2016a}. The proof for the method was modified in \cite{HLM2018a}. Employing the idea in the two original works, some new truncated Euler methods were proposed by using different truncating functions more recently \cite{LX2018a,li2018explicit,YL2018a}. The truncated Milstein method was developed in \cite{GLMY2018}. However, the requirements on the step size for the truncated Milstein method in that work are very restrictive, which brings some difficulties in the applications of the method.
\par
In this paper, we release the constrains on the step size for the Milstein method using some different techniques in the proofs. In addition, the almost sure stability of the Milstein method is investigated.
\par
This paper is organized as follows. Section \ref{secmath} gives the necessary notations and mathematical preliminaries.
The result on the finite time convergence with less constraint step size is presented in section \ref{secconv}. Some simulations are given to illustrate the theoretical results. Section \ref{secas} sees the almost sure stability of the truncated Milstein method. Numerical examples are also used to demonstrate the theorem.

\section{Mathematical Preliminaries}\label{secmath}

Throughout this paper, unless otherwise specified, let $(\Omega , \F, \PP)$ be a complete probability space with a filtration $\left\{\F_t\right\}_{t \ge 0}$ satisfying the usual conditions (that is, it is right continuous and increasing while $\F_0$ contains all $\PP$-null sets). Let $\E$ denote the expectation corresponding to $\PP$.
If $A$ is a vector or matrix, its transpose is denoted by $A^T$.
Let $B(t) = (B^1(t), B^2(t), ..., B^m(t))^T$ be
an $m$-dimensional Brownian motion defined on the space.
If $x \in \RR^d$, then $|x|$ is the Euclidean norm.
For two real numbers $a$ and $b$, set $a\ve b=\max(a,b)$ and $a\we b=\min(a,b)$. If $G$ is a set, its indicator function is denoted by $I_G$, namely $I_G(x)=1$ if $x\in G$ and $0$ otherwise.
\par
Consider a $d$-dimensional SDE
\begin{equation}\label{sde}
dx(t) = \mu(x(t))dt + \s(x(t))dB(t), \quad t\ge 0,
\end{equation}
with the initial value $x(0)=x_0\in\RR^d$, where
$
\mu: \RR^d \to \RR^d,
\quad
\s: \RR^{d\times m} \to\RR^d,
$
and $ x(t)= (x^1(t), x^2(t), ..., x^d(t))^T$.
In some of the proofs in this paper, we need the more specified notation that $\mu=(\mu_1, \mu_2, ..., \mu_d)^T$, $\mu_i: \RR^d \rightarrow \RR$ for $i=1,2,...,d$,
and $\s = (\s_{1,j}, \s_{2,j},...,\s_{d,j})^T$, $\s_{i,j}: \RR^d \rightarrow \RR$ for $j=1,2,...,m$.
For $j_1, j_2 = 1,...,m$, define
\begin{equation}\label{Lg}
L^{j_1}\s_{j_2}(x) = \sum_{l=1}^d \s_{l,j_1}(x)\frac{\partial \s_{j_2} (x)}{\partial x^l} .
\end{equation}
\noindent
For $l= 1,2,...d$, set
\begin{equation*}
\mu'_l(x) = \left( \frac{\partial \mu_l(x)}{\partial x^{1}}, \frac{\partial \mu_l(x)}{\partial x^{2}}, ..., \frac{\partial \mu_l(x)}{\partial x^{d}} \right)~\text{and}~\mu''_l(x) = \left( \frac{\partial^2 \mu_l(x) }{\partial x^j \partial x^i} \right)_{i,j},~i,j=1,2,...,d.
\end{equation*}
\noindent
And for $n= 1,2,...m$, $l = 1,2,...,d$, set
\begin{equation*}
\s'_{l,n}(x) = \left( \frac{\partial \s_{l,n}(x)}{\partial x^{1}}, \frac{\partial \s_{l,n}(x)}{\partial x^{2}}, ..., \frac{\partial \s_{l,n}(x)}{\partial x^{d}} \right)
~\text{and}~\s''_{l,n}(x) = \left( \frac{\partial^2 \s_{l,n}(x) }{\partial x^j \partial x^i} \right)_{i,j},~i,j=1,2,...,d.
\end{equation*}
\noindent
For $j = 1, ..., m$ and $l = 1, ..., d$, define the derivative of the vector $\s_j(x)$ with respect to $x^l$ by
\begin{equation*}
G_j^{l}(x):=\frac{\partial}{\partial x^l} \s_j(x) = \left( \frac{\partial \s_{1,j}(x) }{\partial x^l}, \frac{\partial \s_{2,j}(x) }{\partial x^l}, ...,\frac{\partial \s_{d,j}(x) }{\partial x^l}\right)^T.
\end{equation*}

We impose some standing hypotheses in this paper.
\begin{assp}\label{fgpoly}
There exist constants $K_1 > 0$ and $r > 0$ such that
\begin{equation*}
|\mu(x) - \mu(y)| \ve |\s(x) - \s(y)| \ve \left| L^{j_1} \s_{j_2}(x) -  L^{j_1} \s_{j_2}(y) \right|\leq K_1 (1 + |x|^r + |y|^r) |x - y|
\end{equation*}
for all $x,y \in \RR^d$ and $j_1, j_2=1,2,...m$.
\end{assp}
\begin{assp}\label{KhasminskiiCond}
For every $\overline{p}\ge 1$, there exists a positive constant $K_2$ (dependent on $p$) such that
\begin{equation*}
\left\[ x - y, \mu(x) - \mu(y) \right\] + (2\overline{p}-1)  \left|\s(x) - \s(y) \right|^2 \leq K_2 |x - y|^2
\end{equation*}
for all $x,y \in \RR^d$.
\end{assp}

It is not hard to derive from Assumption \ref{fgpoly} and \ref{KhasminskiiCond}, we can obtain that for all $x \in \RR^d$ and $\l_1\geq 1$
\begin{equation}\label{fgpolyext}
|\mu(x)| \ve |\s(x)| \leq \l_1( 1 + |x|^{r+1}),
\end{equation}
and for any $p\geq \overline{p}\geq 1$
\begin{equation}\label{KhasminskiiCondext}
\left\[ x, \mu(x) \right\] + (2p-1)  \left| \s(x) \right|^2 \leq \l_2(1 + |x|^2),
\end{equation}
 where $\l_2$ is a positive constant dependent on $p$.
\begin{assp}\label{dfgspoly}
 Assume that for $j= 1,2,...m$ and $l = 1,2,...,d$, there exists a positive constant $\l_3$ such that
\begin{equation}
|\mu_l'(x)| \ve |\mu_l''(x)| \ve |\s_{l,j}'(x)| \ve |\s_{l,j}''(x)| \leq \l_3  (1 + |x|^{r+1} ).
\end{equation}
\end{assp}

To define the truncated Milstein method, we first choose a strictly increasing continuous function $\omega : \RR_{+} \rightarrow \RR_{+}$ such that $\omega(u) \rightarrow \infty$ as $u \rightarrow \infty$ and
\begin{equation}\label{formofmu}
\sup_{|x|\leq u} (|\mu(x)| \ve |\s_j(x)| \ve |G_j^{l}(x)| ) \leq \omega(u)
\end{equation}
for any $u \geq 2$, $j = 1, ..., m$ and $l = 1, ..., d$.
Denote the inverse function of $\omega$ by $\omega^{-1}$. We see that $\omega^{-1}$ is a strictly increasing continuous function from $[\omega(0),+\infty)$ to $\RR_{+}$. We also choose a strictly decreasing function
$h: (0,1] \rightarrow (0,+\infty)$ and a constant $\overline{h}\geq 1$ such that
\begin{equation}\label{deltahdelta}
\lim_{\Delta \rightarrow 0}h(\Delta) = \infty ~~\text{and}~~\Delta^{1/4} h(\Delta) \leq \overline{h},~~ \forall \Delta \in (0,1].
\end{equation}
Before we proceed, let us make an useful remark.
\begin{rmk}
In Mao \cite{Mao2015} where the truncated EM was originally developed, it was required to choose a number $\D^*\in (0,1]$ and a strictly decreasing function $h: (0,\D^*] \rightarrow [\omega(0),\8)$
such that
\begin{equation}\label{h1}
h(\Delta^*) \geq \omega(1),~~\lim_{\Delta \rightarrow 0}h(\Delta) = \infty ~~\text{and}~~\Delta^{1/4} h(\Delta) \leq 1,~~ \forall \Delta \in (0,\D^*].
\end{equation}
here, we simply let $\D^*=1$ and remove condition $h(\D^*)\geq \omega(2)$ while we also replace condition $\D^{1/4}h(\D)\leq 1$ by a weaker one $\D^{1/4}h(\D) \leq \widehat{h}$.
In other words, we have made the choice of function $h$ more flexible. We emphasize that such changes to not make any effect on the results in Guo \cite{GLMY2018}. In fact, condition
$h(\D^*)\geq \omega(2)$ was only used to prove [1, Lemma 2.3]. But, in view of Lemma 2.3, we see that the constant $2\a_1$ in [1, Lemma 2.3] is now replaced by another constant $\widehat{K}$ which
does not affect any other results in [1]. It is also easy to check that replacing $\D^{1/4} h(\D) \leq 1$ by $\D^{1/4} h(\D) \leq \widehat{h}$ does not make any effect on
the other results in [1]. Similarly, we see that these change do not affect any result in [1] either.
\end{rmk}

For a given step size $\Delta \in (0,1)$ and any $x \in \RR^d$, let us define a mapping $\pi_\D$ from $\RR^d$ to the closed ball $\{ x\in \RR^d :|x|\leq \omega^{-1}(h(\D))\}$ by
$$
\pi_\D(x)=(|x|\wedge \omega^{-1}(h(\D)) )\frac{x}{|x|},
$$
where we set $x/|x|=0$ when $x=0$. That is, $\pi_\D$ will map $x$ to itself when $|x|\leq \omega^{-1}(h(\D))$ and to $\omega^{-1}(h(\D)) x/|x|$ when $|x| \geq \omega^{-1}(h(\D))$.
We then define the truncated functions, for any $j = 1, ..., m,~l = 1, ..., d$,
\begin{equation*}
\tilde{\mu}(x) = \mu\left( \pi_\D(x) \right),\quad
\tilde{\s}(x) = \s\left( \pi_\D(x) \right)\quad \text{and}
\quad\tilde{G}_j^{l}(x) = G_j^{l}\left( \pi_\D(x) \right).
\end{equation*}
 It is not hard to see that for any $x \in \RR^d$
\begin{equation}\label{hdeltabdfgG}
|\tilde{\mu}(x)| \ve |\tilde{\s}(x)| \ve |\tilde{G}_j^{l}(x)| \leq \omega(\omega^{-1}(h(\Delta))) = h(\Delta).
\end{equation}
That is to say, all the truncated functions $\tilde{\mu}$, $\tilde{\s}$ and $\tilde{G}_j^{l}$ are bounded although $\mu,\s$ and $G_j^{l}$ may not.
Therefore, the truncated Milstein method is defined by
\begin{eqnarray}\label{theTMmethod}
Y_{k+1} \nonumber &=& Y_{k} + \tilde{\mu}(Y_{k}) \Delta +  \tilde{\s}(Y_{k}) \Delta B^j_k
\nonumber \\
&&+ \frac{1}{2} \sum_{j_1= 1}^m \sum_{j_2 = 1}^m  L^{j_1}\tilde{\s}_{j_2}(Y_k)  (\Delta B_k^{j_1} \Delta B_k^{j_2} - \delta _{j_1 j_2} \Delta).
\end{eqnarray}
where $ \delta _{j_1 j_2}=1$ if $j_1=j_2$, else $\delta _{j_1 j_2}=0$. Let us now form two versions of the continuous-time truncated Milstein solutions. The first one is defined by
\begin{equation}
\bar Y(t)=\sum_{k=0}^{\8} Y_{k} I_{[t_k,t_{k+1})}(t), \quad t\geq 0.
\end{equation}
This is simple step process so its sample paths are not continuous. we will refer this as the continuous-time step-process truncated Milstein solution.
The other one is defined by
\begin{equation}\label{continuousTM}
	Y(t) = \bar{Y}(t) + \int_{t_k}^{t} \tilde{\mu}(\bar{Y}(s))ds +  \int_{t_k}^{t} \tilde{\s}(\bar{Y}(s)) dB(s)
+ \sum_{j_1=1}^m \int_{t_k}^t \sum_{j_2=1}^m L^{j_1}\tilde{\s}_{j_2} (\bar{Y}(s)) \Delta B^{j_2}(s)  d B^{j_1}(s),
\end{equation}
where $\Delta B^{j_2}(s) = \sum_{k=0}^{\infty} I_{\{t_k \leq s < t_{k+1} \}} (B^{j_2}(s) - B^{j_2}(t_k))$.

\section{Finite time convergence}\label{secconv}
To point out the restrictive condition imposed in \cite{GLMY2018}, we cites its main result on the convergence rate.
\begin{theorem}\label{T3.1}
(\cite{GLMY2018}). Let Assumptions \ref{fgpoly}, \ref{KhasminskiiCond} and \ref{dfgspoly} hold. Furthermore, assume that for any given $q\geq 1$, there exists a $p\in (q,+\8)$
and a $\D^*$ satisfying (\ref{h1}). In addition, if
\begin{equation}\label{h}
h(\D)\geq \w\big((\D^q(h(\D))^{2q})^{-1/(p-q)} \big)
\end{equation}
holds for all sufficiently small $\D\in (0,\D^*]$, then for any fixed $T=N\D>0$ and sufficiently small $\D\in(0,\D^*]$,
\begin{equation}
\E|x(T)-Y_N|^{2q} \leq K \D^{2q}(h(\D))^{4q}
\end{equation}
holds, where $K$ is a positive constant independent of $\D$.
\end{theorem}

We start this section by presenting our main result.

\begin{theorem}\label{thmmain}
Let Assumptions \ref{fgpoly}, \ref{KhasminskiiCond} and \ref{dfgspoly} hold and assume there exists $q \in [1,p)$ such that
\begin{equation}\label{p}
p>(1+r)q,
\end{equation}
then for any real number $R>|x_0|$ and $\Delta \in (0, 1]$
\begin{equation}\label{strongerror}
\E \left| x(T) - Y_N \right|^{2q} \leq C \big(\Delta^{2q}(h(\Delta))^{4q} \ve(\omega^{-1}(h(\D)))^{-{(2p-2qr-2q)}}\big),
\end{equation}
 where $C$ stands for a generic positive real constant dependent on $p$ but independent of $\Delta$ and its values may change between occurrences.
\end{theorem}

\par
The following example is used to demonstrate the improvement of this new theorem over the convergence result in \cite{GLMY2018}.
\par
\begin{expl}
Consider a scalar SDE
\begin{equation}\label{example1}
dx(t) = (x^3(t) -4 x^5(t))dt + x^2(t)dB(t),~~~t\geq0,
\end{equation}
with the initial value $x(0)=1$.
\end{expl}

It is clear that both of the drift and diffusion coefficients have continuous second-order derivatives. In addition,
it is not hard to verify Assumptions \ref{fgpoly} and \ref{dfgspoly} hold with $r=4$.
For any $x,y\in \RR$ and  $p \geq 1$, we have
\begin{eqnarray*}
& &(x - y) (\mu(x) - \mu(y)) + (2p-1) |\s(x) - \s(y)|^2 \\
&=&(x-y)(x^3-4x^5-y^3+4y^5)+(2p-1)|x^2-y^2|^2\\
&=& (x-y)^2\big[(x^2+xy+y^2)-4(x^4+x^3y+x^2y^2+xy^3+y^4)\big]+(2p-1)(x+y)^2(x-y)^2.
\end{eqnarray*}
Since
\begin{equation*}
-(x^3y+xy^3)=-xy(x^2+y^2)\leq 0.5 (x^2+y^2)^{2}=0.5(x^4+y^4)+x^2y^2,
\end{equation*}
we obtain
\begin{eqnarray*}
& &(x - y) (\mu(x) - \mu(y)) + (2p-1) |\s(x) -\s(y)|^2 \\
&\leq & (x-y)^2[\frac{3}{2}(x^2+y^2)-2(x^4+y^4)+(2p-1)(2x^2+2y^2)]\\
&\leq & (x-y)^2[-2(x^4+y^4)+(4p-\frac{1}{2})(x^2+y^2)]\\
&\leq & 4(p-\frac{1}{8})^2(x-y)^2.
\end{eqnarray*}
That is to say, Assumption \ref{KhasminskiiCond} is also fulfilled. Now, we design the functions $\omega$ and $h$. Noting that
\begin{equation*}
\sup_{|x|\leq u}(|\mu (x)|\ve |\s(x)|\ve|\s'(x)|)\leq 4u^5, \quad \forall u\geq 1,
\end{equation*}
we choose $\omega(u)=4u^5$. Then its inverse function is $\omega^{-1}(u)=(u/4)^{1/5}$. Fix $\varepsilon \in (0,1/4]$, we define $h(\D)=4 \D^{-\varepsilon}$ for $\D \in (0,1)$.
Now, for any $q\geq 1$, we can choose $p$ sufficiently large for
$$
\frac{5}{\varepsilon(2p-2qr-2q)}>2q-4q\varepsilon.
$$
We can therefore conclude by Theorem \ref{thmmain} that the truncated Milstein solution of the SDEs (\ref{example1}) satisfy
$$
\E|x(T)-Y(T)|^{2q} \leq K \D^{2q(1-2\varepsilon)}, \quad \D\in (0, 1].
$$
That is, the strong $L^{2q}$-convergence rate is close to $2q$.

In order to highlight the significant contribution of our new result, let us make a comparison between our new Theorem \ref{thmmain} and one of the main results in \cite{GLMY2018}, namely Theorem \ref{T3.1}.
The key advantage of our new Theorem \ref{thmmain} lies in that it does not need condition (\ref{h}). Let us now explain, via the following example.
\begin{expl}
Consider the scalar SDE
\begin{equation}\label{3.6}
dx(t)=-83x^3(t) dt + x^2(t) dB(t),
\end{equation}
\end{expl}
where $B(t)$ is a scalar Brownian motion. Its coefficients $\mu(x)=-83x^3$ and $\s(x)=x^2$ are clearly locally Lipschitz continuous for $x\in \RR^d$. For $p=42$, we have
$$
x\mu(x)+(2p-1)|\s(x)|^2=0,
$$
so condition (\ref{KhasminskiiCondext}) is satisfied with $p=42$. Moreover, for $\overline{p}=2$, we have
\begin{eqnarray*}
&&(x-y)(\mu(x)-\mu(y))+(2\overline{p}-1)|\s(x)-\s(y)|^2\\
&=&-83(x^2+xy+y^2)|x-y|^2 +3(x^2+2xy+y^2)|x-y|^2\\
&=&-(80x^2+77xy+80y^2)|x-y|^2\\
&\leq& 0.
\end{eqnarray*}
Thus Assumption \ref{KhasminskiiCond} holds with $\overline{p}=2$. Furthermore, it is easy to show that
$$
|\mu(x)-\mu(y)|^2 \ve |\s(x)-\s(y)|^2\ve \left| \s'(x)\s(x) -  \s'(y)\s(y) \right| \leq C(1+x^4+y^4) |x-y|^2.
$$
That is, Assumption \ref{fgpoly} is satisfied with $r=4$.

We first apply Theorem \ref{T3.1} to see what we can get. Obviously, we have $p>\overline{p}$ and we choose $q=1$. Noting
$$
|\mu(x)|\ve |\s(x)| \leq 83|x|^3 \quad \forall |x|\geq 1,
$$
we can then choose $\w(u)=83u^3$ and $h(\D)=\D^{-1/10}$ to define the truncated Milstein solution $Y_N$ to the SDE (\ref{3.6}). It is easy to see that condition
(\ref{h}) becomes
$$
\D^{-1/10} \geq 83 \D^{-12/205}, \quad \text{i.e.}\quad  \D\leq 83^{-410/17}\approx 5.20417 \times 10^{-47}.
$$
For such a small step size, Theorem \ref{T3.1} shows
\begin{equation}\label{3.7}
\E|x(T)-Y_N|^2 \leq C\D^{8/5}.
\end{equation}
The key issue here is that step size is required to be very small, namely less than $5.20417 \times 10^{-47}$, due to condition (\ref{h}).

Let us now apply our new Theorem \ref{thmmain} to see if we can get a better result. We set let $q=1,\w(u)=83u^3$ and $h(\D)=\D^{-1/10}$ as before. Clearly, $p>(1+r)q$.
Noting that $\w^{-1}(u)=(u/83)^{1/3}$ and
$$
(\w^{-1}(h(\D)))^{-(2p-2qr-2q)}+\D^{2q}(h(\D))^{4q}= 83^{74/3} \D^{74/30} \ve \D^{8/5} =O(\D^{8/5})
$$
we conclude by Theorem \ref{thmmain} that for any $\D\in(0,1]$
$$
\E|x(T)-Y_N|^2 \leq C\D^{8/5}
$$
That is the same as (\ref{3.7}) but the step size $\D$ can now be any number in $(0,1]$ rather than $\D\leq 5.20417 \times 10^{-47}$.


In the computer, we choose $\varepsilon=0.1$ and regard the numerical solution with step size of 0.01 as the true solution. In Fig 1, we plot the strong errors of the truncated Milstein method with step size
$2\times 0.01,2^2\times 0.01,2^3\times 0.01,2^4\times 0.01,2^5\times 0.01$, and $2^6\times 0.01$, respectively. Clearly, we can see the strong convergence rate is close to one.
\begin{figure}
	\centering
    \includegraphics[height=6cm, width=8cm]{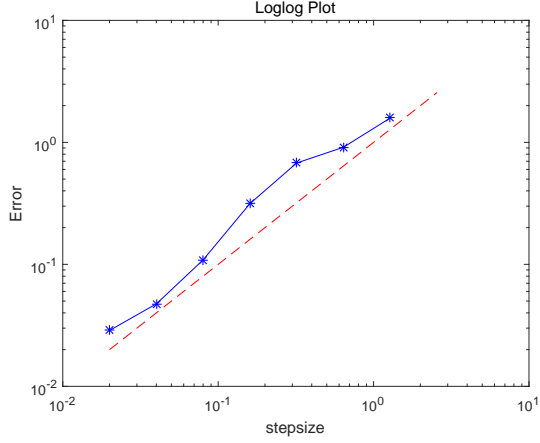}
	\caption{Loglog plot of the errors against the step sizes. The red dash line is of slope 1 and blue line indicates errors}
\end{figure}
\par
To prove the main result, we need to prepare some necessary lemmas. Due to the proofs of these lemmas are either quite standard or closely following those in \cite{GLMY2018}, we put them in Appendix \ref{appendixsec}.
\par \noindent
We are ready to give the proof of the main result.
\par \noindent
{\it Proof of Theorem \ref{thmmain}}
\par
To obtain assertion (\ref{strongerror}), we first to proof the following assertion
\begin{equation}\label{strongerror1}
\E \left( \left| x(t\wedge \theta)-Y(t\wedge \theta) \right|^{2q}  \right)\leq C \big(\Delta^{2q}(h(\Delta))^{4q} \ve(\omega^{-1}(h(\D)))^{-{(2p-2qr-2q)}} \big)
\end{equation}
hold, where $\theta:=\inf\{t\geq 0: |x(t)|\ve |Y(t)|\geq R \}$ is stopping time.
 By the It\^{o} formula, we can show that for any $0\leq t\leq T$,
\begin{equation}\label{eq:ito1}
\begin{split}
&\E \left( \left| x(t\wedge \theta)-Y(t\wedge \theta) \right|^{2q}  \right)\\
&= 2q \E \int_0^{t\wedge \theta}  \left| x(s)-Y(s) \right|^{2q-2} \Big\langle x(s) - {Y}(s),  \mu(x(s))-\tilde{\mu}(\bar{Y}(s)) \Big\rangle ds\\
&\quad+ 2q\sum_{i=1}^m \E \int_0^{t\wedge \theta} \left| x(s)-Y(s) \right|^{2q-2} \times\\
&\quad \quad\quad \frac{2q-1}{2}\Big| \s_i(x(s)) - \tilde{\s}_i(\bar{Y}(s))-\sum_{j=1}^m L^j  \tilde{\s_i}(\bar{Y}(s)) \Delta B^j(s) \Big|^2 ds.
\end{split}~~~~
\end{equation}
By plus-and-minus technique, we have
$$
 \mu(x(s))-\tilde{\mu}(\bar{Y}(s)) =\mu(x(s))-\mu(Y(s))+\mu(Y(s))-\tilde{\mu}(Y(s))+\tilde{\mu}(Y(s))-\tilde{\mu}(\bar{Y}(s)),
$$
and
\begin{equation*}
\begin{split}
&\frac{2q-1}{2} \Big|\s_i(x(s)) - \tilde{\s}_i(\bar{Y}(s))-\sum_{j=1}^m L^j  \tilde{\s_i}(\bar{Y}(s)) \Delta B^j(s)\Big|^2\\
=&\frac{2q-1}{2} \Big|\s_i(x(s)) -\s_i(Y(s)) +\s_i(Y(s))-\tilde{\s}_i({Y}(s))\\
&\quad +\tilde{\s}_i({Y}(s))-\tilde{\s}_i(\bar{Y}(s))-\sum_{j=1}^m L^j  \tilde{\s_i}(\bar{Y}(s)) \Delta B^j(s)\Big|^2\\
\leq& \frac{2p-1}{2} \Big|\s_i(x(s)) -\s_i(Y(s))\Big|^2+ \frac{(2p-1)(2q-1)}{2p-2q} \Big|\s_i(Y(s))-\tilde{\s}_i({Y}(s))\Big|^2\\
&\quad +\frac{(2p-1)(2q-1)}{2p-2q} \Big|\tilde{\s}_i({Y}(s))-\tilde{\s}_i(\bar{Y}(s))-\sum_{j=1}^m L^j  \tilde{\s_i}(\bar{Y}(s)) \Delta B^j(s)\Big|^2\\
\leq& ({2p-1}) \Big|\s_i(x(s)) -\s_i(Y(s))\Big|^2+ \frac{(2p-1)(2q-1)}{2p-2q} \Big|\s_i(Y(s))-\tilde{\s}_i({Y}(s))\Big|^2\\
&\quad +\frac{(2p-1)(2q-1)}{2p-2q} \Big|\tilde{\s}_i({Y}(s))-\tilde{\s}_i(\bar{Y}(s))-\sum_{j=1}^m L^j  \tilde{\s_i}(\bar{Y}(s)) \Delta B^j(s)\Big|^2.
\end{split}
\end{equation*}
Therefore, it follows from \eqref{eq:sigmaTaylor} and (\ref{eq:ito1}) that
\begin{equation}\label{eq:ito2}
\begin{split}
&	\E \left( \left| x(t\wedge \theta)-Y(t\wedge \theta) \right|^{2q}  \right)\\
= & 2q \E \int_0^{t\wedge \theta}  \left| x(s)-Y(s) \right|^{2q-2} \big\langle x(s) - {Y}(s),  \mu(x(s))-{\mu}({Y}(s)) \big\rangle  ds\\
&\quad+ 2q\sum_{i=1}^m \E \int_0^{t\wedge \theta} \left| x(s)-Y(s) \right|^{2q-2} (2q-1)\Big| \s_i(x(s)) - {\s}_i({Y}(s)) \Big|^2 ds\\
& +2q \E \int_0^{t\wedge \theta}  \left| x(s)-Y(s) \right|^{2q-2} \big\langle x(s) - {Y}(s),  \mu(Y(s))-\tilde{\mu}({Y}(s)) \big\rangle  ds\\
& +2q \E \int_0^{t\wedge \theta}  \left| x(s)-Y(s) \right|^{2q-2} \big\langle x(s) - {Y}(s),  \tilde{\mu}({Y}(s))-\tilde{\mu}(\bar{Y}(s)) \big\rangle  ds\\
&+ 2q\sum_{i=1}^m \E \int_0^{t\wedge \theta} \frac{(2p-1)(2q-1)}{2p-2q}\left| x(s)-Y(s) \right|^{2q-2}\Big| \s_i(Y(s)) - \tilde{\s}_i({Y}(s)) \Big|^2 ds\\
&+ 2q\sum_{i=1}^m \E \int_0^{t\wedge \theta} \frac{(2p-1)(2q-1)}{2p-2q}\left| x(s)-Y(s) \right|^{2q-2}  \Big|\tilde{R}_1(\tilde{\s}_i) \Big|^2 ds\\
\leq & 2q \big(J_1+J_2+J_3+J_4+J_5\big),
\end{split}
\end{equation}
where
\begin{eqnarray}\label{IntJ1}
J_1&=& \E \int_0^{t\wedge \theta}  \left| x(s)-Y(s) \right|^{2q-2}\Big( \big\langle x(s) - {Y}(s),  \mu(x(s))-{\mu}({Y}(s)) \big\rangle  ds\\
\nonumber
&\quad&+(2q-1)\sum_{i=1}^m\Big| \s_i(x(s)) - {\s}_i({Y}(s)) \Big|^2 \Big) ds,\\
\label{IntJ2}
 J_2&=&\E \int_0^{t\wedge \theta}  \left| x(s)-Y(s) \right|^{2q-2} \big\langle x(s) - {Y}(s), \mu(Y(s))-\tilde{\mu}({Y}(s)) \big\rangle  ds,\\
\label{IntJ3}
 J_3&=&\E \int_0^{t\wedge \theta}  \left| x(s)-Y(s) \right|^{2q-2} \big\langle x(s) - {Y}(s),  \tilde{\mu}({Y}(s))-\tilde{\mu}(\bar{Y}(s)) \big\rangle  ds,\\
\label{IntJ4}
 J_4&=& \frac{(2p-1)(2q-1)}{2p-2q}\sum_{i=1}^m \E \int_0^{t\wedge \theta} \left| x(s)-Y(s) \right|^{2q-2}\Big| \s_i(Y(s)) - \tilde{\s}_i({Y}(s)) \Big|^2 ds,
\end{eqnarray}
and
\begin{equation}\label{IntJ5}
J_5= \frac{(2p-1)(2q-1)}{2p-2q}\sum_{i=1}^m \E \int_0^{t\wedge \theta} \left| x(s)-Y(s) \right|^{2q-2}\Big|\tilde{R}_1(\tilde{\s}_i) \Big|^2 ds.~~~~~~~~~~~~~~~~~
\end{equation}
Applying Assumption \ref{KhasminskiiCond} to $J_1$, we obtain
\begin{equation}\label{Est_J1}
J_1\leq K_2 \E \int_0^{t\wedge \theta}  \left| x(s)-Y(s) \right|^{2q}  ds.
\end{equation}
Inserting the expression (\ref{taylorformula2}) into (\ref{IntJ2}) gives
\begin{equation}
\begin{split}
J_2&=\E \int_0^{t\wedge \theta}  \left| x(s)-Y(s) \right|^{2q-2} \big\langle x(s) - {Y}(s),  \mu(Y(s))-\tilde{\mu}({Y}(s)) \big\rangle  ds\\
&=\E \int_0^{t\wedge \theta}  \left| x(s)-Y(s) \right|^{2q-2} \big\langle x(s) - {Y}(s),  \mu(Y(s))-{\mu}(\pi_{\D}({Y}(s))) \big\rangle  ds\\
&\leq \E \int_0^{t\wedge \theta}  \left| x(s)-Y(s) \right|^{2q-2} \times\\
&\quad\quad\quad \big\langle x(s) - {Y}(s),  \mu'(\pi_{\D}({Y}(s)))\big(\sum_{j=1}^m \int_{t_k}^{s} \tilde{\s}_j(\pi_{\D}({Y}(s_1))) dB^j(s_1)\big) + \tilde{R}_1(\mu) \big\rangle  ds.
\end{split}
\end{equation}
By the Young inequality and Holder inequality, we get
\begin{equation}
J_2\leq  C \E \int_0^{t\wedge \theta} \big(\left| x(s)-Y(s) \right|^{2q}+|\tilde{R}_1(\mu)|^{2q}\big) ds  +  C I_1,
\end{equation}
where
$$
I_1= \E \int_0^{t\wedge \theta} \big| \big\langle x(s) - {Y}(s),  \mu'(\pi_{\D}({Y}(s)))\big(\sum_{j=1}^m \int_{t_k}^{s} {\s}_j(\pi_{\D}({Y}(s_1))) dB^j(s_1)\big)  \big\rangle \big|^{q}  ds.
$$
Following a very similar approach used for (3.23) in \cite{WG2013}, we can show that
\begin{equation*}
I_1\leq C \Delta^{2q}.
\end{equation*}
Therefore, we have
\begin{equation}\label{Est_J2}
\begin{split}
J_2 \leq C \E \int_0^{t\wedge \theta} \big( \left| x(s) - {Y}(s) \right|^{2q}  +  \big| \tilde{R}_1(\mu)  \big|^{2q}\big)  ds +C \Delta^{2q}.
\end{split}
\end{equation}
Similarly to $J_2$, using the Young inequality we obtain that
\begin{equation}\label{Est_J3}
\begin{split}
J_3&=\E \int_0^{t\wedge \theta}  \left| x(s)-Y(s) \right|^{2q-2} \big\langle x(s) - {Y}(s),  \tilde{\mu}({Y}(s))-\tilde{\mu}(\bar{Y}(s)) \big\rangle  ds\\
&\leq C \E \int_0^{t\wedge \theta} \big( \left| x(s) - {Y}(s) \right|^{2q}  +  \big| \tilde{R}_1(\mu)  \big|^{2q}\big)  ds +C \Delta^{2q}.
\end{split}
\end{equation}
Applying the Young inequality to (\ref{IntJ4}) gives
\begin{equation}
\begin{split}
J_4&=\frac{(2p-1)(2q-1)}{2p-2q}\sum_{i=1}^m \E \int_0^{t\wedge \theta} \left| x(s)-Y(s) \right|^{2q-2}\Big| \s_i(Y(s)) - \tilde{\s}_i({Y}(s)) \Big|^2 ds\\
 &\leq C \E\int_0^{t\wedge \theta} \left|x(s) - {Y}(s) \right|^{2q} ds +C\sum_{i=1}^m \E\int_0^{t\wedge \theta} \left| \s_i(Y(s)) - \tilde{\s}_i({Y}(s))  \right|^{2q} ds\\
 &\leq C \E\int_0^{t\wedge \theta} \left|x(s) - {Y}(s) \right|^{2q} ds +CI_2,
\end{split}
\end{equation}
where
$$
I_2=\sum_{i=1}^m \E\int_0^{t\wedge \theta} \left| \s_i(Y(s)) - \tilde{\s}_i({Y}(s))  \right|^{2q} ds.
$$
By the Assumption \ref{fgpoly} and Lemma \ref{TMbound}, we derive that
\begin{equation}
\begin{split}
I_2&=\sum_{i=1}^m \E\int_0^{t\wedge \theta} \left| \s_i(Y(s)) -{\s}_i(\pi_{\D}({Y}(s)))  \right|^{2q} ds\\
&\leq \sum_{i=1}^m \E\int_0^{t\wedge \theta}(1+|Y(s)|^{2qr}+|\pi_{\D}(Y(s))|^{2qr}) \left| Y(s) -\pi_{\D}({Y}(s))  \right|^{2q} ds\\
&\leq  \sum_{i=1}^m\int_0^{T}\E(1+|Y(s)|^{2p}+|\pi_{\D}(Y(s))|^{2p})^{\frac{2qr}{2p}}\Big(\E \left| Y(s) -\pi_{\D}({Y}(s))  \right|^{\frac{2p\cdot2q}{2p-2qr}}\Big)^{\frac{2p-2qr}{2p}} ds\\
&\leq C\sum_{i=1}^m\int_0^{T}\Big(\E \left| Y(s) -\pi_{\D}({Y}(s))  \right|^{\frac{2p\cdot2q}{2p-2qr}}\Big)^{\frac{2p-2qr}{2p}} ds\\
&\leq C\sum_{i=1}^m\int_0^{T}\Big( \E[I_{\{|Y(s)|>\omega^{-1}(h(\D))\}}|Y(s)|^{\frac{2p\cdot2q}{2p-2qr}}]\Big)^{\frac{2p-2qr}{2p}}  ds\\
&\leq C\sum_{i=1}^m\int_0^{T} \Big([\mathbb{P}\{|Y(s)>\omega^{-1}(h(\D))|\}]^{\frac{2p-2qr-2q}{2p-2qr}}[{\E|Y(s)|^{2p}}]^{\frac{2q}{2p-2qr}}\Big)^{\frac{2p-2qr}{2p}}ds\\
&\leq C\sum_{i=1}^m\int_0^{T} \Big(\frac{\E|Y(s)|^{2p}}{(\omega^{-1}(h(\D)))^{2p}}\Big)^{\frac{2p-2qr-2q}{2p}} ds\\
& \leq C(\omega^{-1}(h(\D)))^{-{(2p-2qr-2q)}}.
\end{split}
\end{equation}
Therefore, we can obtain
\begin{equation}\label{Est J4}
J_4\leq C \E\int_0^{t\wedge \theta} \left|x(s) - {Y}(s) \right|^{2q} ds+C(\omega^{-1}(h(\D)))^{-{(2p-2qr-2q)}}.
\end{equation}
We also using the Young inequality to $J_5$ get
\begin{equation}\label{Est J5}
\begin{split}
J_5&=\frac{(2p-1)(2q-1)}{2p-2q}\sum_{i=1}^m \E \int_0^{t\wedge \theta} \left| x(s)-Y(s) \right|^{2q-2}\Big|\tilde{R}_1(\tilde{\s}_i) \Big|^2 ds\\
&\leq C\sum_{i=1}^m \E \int_0^{t\wedge \theta} \left|x(s) - {Y}(s) \right|^{2q} ds+ C\sum_{i=1}^m \E \int_0^{t\wedge \theta} \Big|\tilde{R}_1(\tilde{\s}_i) \Big|^{2q} ds\\
&\leq C\sum_{i=1}^m \E \int_0^{t\wedge \theta} \left|x(s) - {Y}(s) \right|^{2q} ds+ C  \Delta^{2q}(h(\Delta))^{4q}.
\end{split}
\end{equation}
Substituting (\ref{Est_J1}), (\ref{Est_J2}), (\ref{Est_J3}), (\ref{Est J4}) and (\ref{Est J5}) into (\ref{eq:ito2}), and then applying the Gronwall inequality and Lemma \ref{Rlem},
\begin{equation}
\begin{split}
\E \left( \left| x(t\wedge \theta)-Y(t\wedge \theta) \right|^{2q}  \right)\leq C \big(\Delta^{2q}(h(\Delta))^{4q} \ve(\omega^{-1}(h(\D)))^{-{(2p-2qr-2q)}}\big),
\end{split}
\end{equation}
which is assertion (\ref{strongerror1}). Finally,
using the well-known Fatou lemma, we can let $R\rightarrow \8$ to obtain the desired assertion (\ref{strongerror}). \eproof

\section{Almost sure stability} \label{secas}
In the section we discuss the preservation of the almost sure asymptotic stability of the underlying SDEs (\ref{sde}) by using the truncated Milstein method. To study the stability, we also assume that
$$
\mu(0)=0,\quad \s(0)=0.
$$

To guarantee the almost sure asymptotic stability of the underlying SDEs (\ref{sde}), we need an additional assumption.

\begin{assp}\label{A4.1}
Assume that there exists a function $k\in\mathcal{K}$ such that
\begin{equation}
2x^T \mu(x)+|\s(x)|^2 \leq -k(|x|)
\end{equation}
for all $x\in \RR^d$, where $\mathcal{K}$ denotes the family of continuous nondecreasing functions $k : \RR_+\rightarrow \RR_+ $ such that $k(0)=0$ and $k(u)>0$ for all $u>0$..
\end{assp}

The following theorem from \cite{Mao2002a} states the almost sure asymptotic stability of the underlying SDEs.
\begin{theorem}\label{T4.1}
Let Assumption \ref{A4.1} hold. Then for any initial value $x_0\in \RR^d$, the solution to the SDE (\ref{sde}) satisfies
\begin{equation}
\lim_{t\rightarrow \8} x(t)=0 \quad a.s.
\end{equation}
\end{theorem}

The following theorem shows that the truncated Milstein method can preserve this almost surely asymptotical stability with an addition condition (\ref{4.3}).
\begin{theorem}
Let  Assumption \ref{A4.1} hold, and for any $\D\in(0, 1]$
\begin{equation}\label{-K}
2 x^T {\mu}(x) +| {\s}(x)|^2+\frac{1}{2}|\sum_{j_1= 1}^m \sum_{j_2 = 1}^m L^{j_1}  \s_{j_2} (x) |^2 \D \leq -k(|x|)
\end{equation}
hold.  Assume also that
\begin{equation}\label{4.3}
\limsup_{|x|\rightarrow 0}\frac{|\mu(x)|^2}{k(|x|)}<\8.
\end{equation}
Set
\begin{equation}\label{H}
H=\sup_{0<|x|<\omega^{-1}(h(1))}\frac{|\mu(x)|^2}{k(|x|)},
\end{equation}
and
\begin{equation}\label{4.4}
\D_1=\min\Big(1,    0.5/H,   0.25(k(\omega^{-1}(h(1))))^2 \Big).
\end{equation}
Then for any $\D\in(0,\D_1]$ and any initial value $x_0\in \RR^d$, the solution of the truncated Milstein method (\ref{theTMmethod}) satisfies
\begin{equation}\label{4.7}
\lim_{k\rightarrow \8}Y_k =0 \quad  a.s.
\end{equation}
\end{theorem}

\begin{proof}
We first observe that $H< \8$ from condition (\ref{H}), hence we have $\D_1 \in (0,1]$.
Next, We show that the truncated functions $\tilde \mu$ and $\tilde \s$ preserve property (\ref{-K}) perfectly in the sense that,
\begin{equation}
2 x^T \tilde{\mu}(x) +| \tilde{\s}(x)|^2+\frac{1}{2}|\sum_{j_1= 1}^m \sum_{j_2 = 1}^m L^{j_1} \tilde \s_{j_2} (x) |^2 \D \leq -k(|x|).
\end{equation}
According to the $\frac{1}{2}|\sum_{j_1= 1}^m \sum_{j_2 = 1}^m L^{j_1} \tilde \s_{j_2} (x) |^2 \D\ge 0$ and using the same technique in \cite{HLM2018a}, so we omit the proof.

Let now fix any $\D \in(0,1]$ and $x_0\in \RR^d$. Squaring both sides of the (\ref{theTMmethod}), we are easy to arrive at
\begin{equation}\label{square}
\begin{split}
|Y_{k+1}|^2
&=|Y_k|^2+2 Y_k^T \tilde{\mu}(Y_k) \D+| \tilde{\s}(Y_k)|^2\D+\frac{1}{2}|\sum_{j_1= 1}^m \sum_{j_2 = 1}^m L^{j_1} \tilde \s_{j_2} (Y_k) |^2 \D^2\\
&\quad+|\tilde{\mu}(Y_k)|^2\D^2 +m_{k+1},
\end{split}
\end{equation}
where
\begin{equation}
\begin{split}
m_{k+1}=& | \tilde{\s}(Y_k)|^2(\Delta B_k^{j_1} \Delta B_k^{j_2}-\Delta)
 +\frac{1}{4}|\sum_{j_1= 1}^m \sum_{j_2 = 1}^m L^{j_1} \tilde \s_{j_2} (Y_k) |^2[(\Delta B_k^{j_1} \Delta B_k^{j_2}-\Delta)^2-2\D^2]\\
&+2\big\langle Y_k, \tilde{\s}(Y_k)\D B_k +\frac{1}{2} \sum_{j_1= 1}^m \sum_{j_2 = 1}^m L^{j_1} \tilde \s_{j_2} (Y_k)\big( \Delta B_k^{j_1} \Delta B_k^{j_2}-\Delta\big)\big\rangle\\
&+ 2\big\langle \tilde{\mu}(Y_k)\D,  \tilde{\s}(Y_k)\D B_k +\frac{1}{2} \sum_{j_1= 1}^m \sum_{j_2 = 1}^m L^{j_1} \tilde \s_{j_2} (Y_k)\big( \Delta B_k^{j_1} \Delta B_k^{j_2}-\Delta\big)  \big\rangle\\
&+   \big\langle 2 \tilde{\s}(Y_k)\D B_k, \sum_{j_1= 1}^m \sum_{j_2 = 1}^m L^{j_1} \tilde \s_{j_2} (Y_k)\big( \Delta B_k^{j_1} \Delta B_k^{j_2}-\Delta\big) \big\rangle
\end{split}
\end{equation}
is a local martingale difference.

Following a same approach used for (5.14) in \cite{HLM2018a}, we can show
\begin{equation}
|\tilde{\mu}(x)|^{2}\D \leq 0.5 k(|\pi_{\D}(x)|),
\end{equation}
Substituting this into (\ref{square}) and according to the (\ref{-K}), we get
\begin{equation}
\begin{split}
|Y_{k+1}|^2
&=|Y_k|^2-0.5\D k(|\pi_\D(Y_k)|) +m_{k+1}.
\end{split}
\end{equation}
This implies
\begin{equation}
\begin{split}
|Y_{k+1}|^2
&=|Y_0|^2-0.5\D\sum_{i=0}^k k(|\pi_\D(Y_i)|) +\sum_{i=0}^k m_{i+1}.
\end{split}
\end{equation}
Applying the nonnegative semi-martingale convergence theorem, we get
$$
\sum_{i=0}^{\8} k(|\pi_\D(Y_i)|) < \8 \quad a.s.
$$
This implies
$$
\lim_{i\rightarrow \8}k(|\pi_\D(Y_i)|) =0 \quad a.s.
$$
\end{proof}
Consequently, we must have
$$
\lim_{i\rightarrow \8}|\pi_\D(Y_i)| =0 \quad a.s.
$$
and the desired assertion (\ref{4.7}) follows. The proof is therefore complete. \eproof
\begin{expl}
Consider a scalar SDE
\begin{equation}\label{4.14}
dx(t) = (-x(t)-6x^3(t) -4 x^5(t))dt + x^2(t)dB(t),~~~t\geq0,
\end{equation}
with the initial value $x(0)=1$.
\end{expl}
It can be seen that for any $\D \in (0,1]$
\begin{eqnarray*}
2x^T\mu(x)+|\s(x)|^2 +\frac{1}{2}|\s'(x)\s(x)|^2\D
&=&2x(-x-6x^3-4x^5)+|x^2|^2+\frac{1}{2}|2x\cdot x^2|^2\D\\
&=&-2x^2 -12x^4-8x^6+x^4+2x^6\D \\
&\leq &-2x^2 -11x^4-6x^6\leq -2x^2.
\end{eqnarray*}
Let us choose $\omega(u)=4u^5$ and $h(\D)=4\D^{-1/4}$. It is not hard to see that (\ref{H}) is satisfied with $k(u)=2u^2$. Moreover, we observe
$$
\limsup_{|x|\rightarrow 0} \frac{|\mu(x)|^2}{k(|x|)}=\limsup_{|x|\rightarrow 0} \frac{|-x-6x^3-4x^5|^2}{2|x|^2} < \8,
$$
which implies that (\ref{4.3}) holds. Noting that $\omega^{-1}(h(1))=1$,
we can compute by (\ref{H}) and (\ref{4.4}) that $H=25$ and $\D_1=0.04$.
According to Theorem \ref{T4.1}, now we can conclude that for every $\D\in (0,0.04]$ and any initial value $x_0\in \RR$ the truncated Milstein method satisfies
\begin{equation}
\lim_{k\rightarrow \8}Y_k =0 \quad  a.s.
\end{equation}

Fig 2 displays 10 paths of solutions generated by the truncated Milstein method. It can be seen that the almost sure stability of SDE (\ref{4.14}) is preserved.
\begin{figure}
	\centering
    \includegraphics[height=6cm, width=8cm]{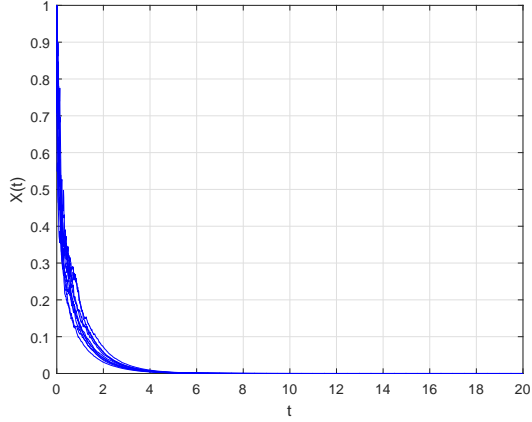}
	\caption{10 paths of the numerical solutions generated by the truncated Milstein method}
\end{figure}

\appendix
\section{Useful lemmas} \label{appendixsec}

The first one is the standard result on the moments bound of the underlying solution. The proof could be found in, for example \cite{M2008a}.
\begin{lemma}
Under Assumption \ref{KhasminskiiCond}, there exists a positive constant $K$, dependent on $t$ and $p$, such that
 \begin{equation*}\label{sdepthmoment}
\E |x(t)|^{2p} \leq K \left( 1 + |x(0)|^{2p} \right).
\end{equation*}
\end{lemma}
\par \noindent
Since the main change in the condition in this paper is the first inequality in \eqref{deltahdelta}, Lemmas \ref{preservecon} to \ref{Rlem} could be proved by closely following those approaches in \cite{GLMY2018}. Therefore, we omit proofs of them here and only detail the proof of Theorem \ref{thmmain}, in which some different techniques are used to release the constrains on the step size.
\par
The following Lemma shows that the functions $\tilde \mu $ and $\tilde \s$ preserve (\ref{KhasminskiiCondext}) for all $\Delta \in (0, \Delta^*]$.
\begin{lemma} \label{preservecon}
Assume that (\ref{KhasminskiiCondext}) holds.  Then, for all $\overline{p}\geq 1$ and any $x \in \RR^d$,
	\begin{equation}\label{truncatedKhasminskiiCondext}
	\left\[ x, \tilde{\mu}(x) \right\] + (2\overline{p}-1) \left| \tilde{\s}(x) \right|^2 \leq 2\l_2(1 + |x|^2).
	\end{equation}
\end{lemma}
 The next lemma presents the moment bound of the numerical solution.
\begin{lemma}\label{TMbound}
Let Assumptions \ref{fgpoly}, \ref{KhasminskiiCond} and \ref{dfgspoly} hold. Then for any $\Delta \in (0, \Delta^*]$ and any $T > 0$, $p\geq 1$
\begin{equation*}
\sup_{0 < \Delta \leq \Delta^*}\sup_{0 \leq t \leq T}\E |Y(t)|^{2p} \leq K \left( 1 + \E |Y(0)|^{2p} \right),
\end{equation*}
where K is a positive constant dependent on $T$ but independent of $\Delta$.
\end{lemma}
The lemma below gives the difference between the two continuous versions of the truncated Milstein method.
\begin{lemma}\label{YYbar}
	For any $\Delta \in (0,\Delta^{*}]$, and any $p \geq 1$,
	\begin{equation*}
	\E |Y(t) - \bar{Y}(t)|^{2p} \leq C \Delta^{p}(h(\Delta))^{2p},
	\end{equation*}
	where $C$ is a positive constant independent of $\Delta$.
\end{lemma}

Comparing Lemmas \ref{sdepthmoment} and \ref{TMbound} with those assumptions in Section \ref{secmath}, we obtain the next two lemmas.

\begin{lemma} \label{mblem0}
If Assumptions \ref{fgpoly}, \ref{KhasminskiiCond} and \ref{dfgspoly} hold, then for all $p\geq 1$ and $j_1,j_2 = 1, ..., m$,
\begin{equation}
\sup_{0 < \Delta \leq \Delta^*}\sup_{0 \leq t \leq T} \left[ \mathbb{E}|\mu(Y(t))|^{p} \vee \mathbb{E} |\mu'(Y(t))|^{p} \vee \mathbb{E} |\s(Y(t))|^{p} \vee \mathbb{E} |L^{j_1} \s_{j_2}(Y(t))|^p\right] < \infty.
\end{equation}
\end{lemma}

\begin{lemma} \label{mblem}
	If Assumptions \ref{fgpoly} and \ref{KhasminskiiCond} hold, then for all $p \geq 1$ and $j = 1, ..., m$,
	\begin{equation} \label{d7}
	\sup_{0 \leq t \leq T} \left[ \mathbb{E}|x(t))|^{p} \vee \mathbb{E} |\mu(x(t))|^{p} \vee \mathbb{E} |\s_j(x(t))|^{p} \right]<\infty.
\end{equation}
\end{lemma}

Let us brief a version of the deterministic Taylor formula.
 If a function  $f:\mathbb{R}^d \rightarrow \mathbb{R}^d$ is twice differentiable, then the following Taylor formula
\begin{equation}\label{taylorformula1}
	\begin{split}
		&f(x)-f(x^*) = f'(x^*)(x-x^*) + R_1(f)
	\end{split}
\end{equation}
holds, where $R_1(f)$ is the remainder term
\begin{equation}\label{R1}
	\begin{split}
		R_1(f) =& \int_0^1(1-\varsigma) f''(x^*+\varsigma(x-x^*))(x-x^*,x-x^*) d\varsigma.
	\end{split}
\end{equation}
For any $x, h_1,h_2 \in \mathbb{R}^d$, the derivatives have the following expressions
\begin{equation}\label{derivative}
	f'(x)(h_1) = \sum_{i=1}^{d} \frac{\partial f}{\partial x^{i}}h_1^{i}, \quad
	f''(x)(h_1,h_2) = \sum_{i,j=1}^{d} \frac{\partial^2 f}{\partial x^{i}\partial x^{j}}h_1^{i}h_2^{j}.
\end{equation}
Here,
\begin{equation*}
\frac{\partial f}{\partial x^{i}} = \left(\frac{\partial f_1}{\partial x^{i}}, \frac{\partial f_2}{\partial x^{i}}, ..., \frac{\partial f_d}{\partial x^{i}}  \right), \quad f = (f_1, f_2, ..., f_d).
\end{equation*}
Replacing $x$ and $x^*$ in \eqref{taylorformula1} by $Y(t)$ and $\bar{Y}(t)$, respectively, from \eqref{continuousTM} we have
\begin{equation}\label{taylorformula2}
	\begin{split}
		f(Y(t))-f(\bar{Y}(t))  = f'(\bar{Y}(t))\big(\sum_{j=1}^m \int_{t_k}^{t} \tilde{\s}_j(\bar{Y}(s)) dB^j(s))\big) + \tilde{R}_1(f),
	\end{split}
\end{equation}
where
\begin{equation}\label{RT1}
	\begin{split}
		\tilde{R}_1(f)=f'(\bar{Y}(t)) \Big(\int_{t_k}^{t} \tilde{\mu}(\bar{Y}(s))ds + \sum_{j_1=1}^m \int_{t_k}^t \sum_{j_2=1}^m L^{j_1}\tilde{\s}_{j_2} (\bar{Y}(s)) \Delta B^{j_2}(s)  d B^{j_1}(s)\Big) +R_1(f).
	\end{split}
\end{equation}
By (\ref{Lg}) and (\ref{derivative}), we find
\begin{equation}\label{eq:sigmac}
	\tilde{\s}_{i}'(x) \big(\tilde{\s}_{j}(x) \big)= L^{j}\tilde{\s}_{i}(x).
\end{equation}
Therefore, by \eqref{eq:sigmac},  replacing $f$ in \eqref{taylorformula2} by $\s_{i}$ gives
\begin{equation} \label{eq:sigmaTaylor}
	\begin{split}
		\tilde{R}_1(\s_{i}) = \s_{i}({Y}(t))-\s_{i}(\bar{Y}(t))-\sum_{j=1}^{m}L^{j}\s_{i}(\bar{Y}(t))\Delta B^{j}(t)
	\end{split}
\end{equation}
for $t_k\leq t<t_{k+1}$.
\par
Now we give some estimates on the residues.
\begin{lemma} \label{Rlem}
If Assumptions \ref{fgpoly}, \ref{KhasminskiiCond} and \ref{dfgspoly} hold, then for $i =1,2,..., m$ and all $p \geq 1$,
	\begin{equation} \label{R_estimate}
		\E|\tilde{R}_1(\mu) |^p \vee \E|\tilde{R}_1(\s_{i})|^p \vee \E|\tilde{R}_1(\tilde \s_{i})|^p \leq C \Delta^p(h(\Delta))^{2p},
	\end{equation}
	where $C$ is a positive constant independent of $\Delta$.
\end{lemma}

\section*{Acknowledgements}

The authors would like to thank
the National Natural Science Foundation of China (11701378), ``Chenguang Program'' supported by both Shanghai Education Development Foundation
and Shanghai Municipal Education Commission (16CG50),and Shanghai Pujiang Program (16PJ1408000)
for their financial support.


\end{document}